\documentclass[12pt]{article}
\usepackage{amsmath,amsfonts,amsthm,amssymb,enumerate,hyperref,stmaryrd}
\usepackage{young}
\usepackage[...]{youngtab}
\usepackage[margin=2.5cm]{geometry}
\usepackage[all]{xy}
\usepackage{listings}
\usepackage[normalem]{ulem}
\usepackage{color}
\usepackage{framed}
\usepackage{comment}
\definecolor{shadecolor}{gray}{0.875}
\renewcommand{\P}{\mathbb{P}}

\renewcommand{\O}{\mathcal{O}}

\newcommand{\C}{\mathbb{C}}
\newcommand{\K}{\mathbb{K}}
\newcommand{\ra}{\rightarrow}

\def\Kp1{{{\rm K}}_{p,1}(\mathbb P^n,b;d)}
\def\Bpq{{{\mathbb K}}_{p,q}(b;d)}
\def\Bp1{{{\mathbb K}}_{p,1}(b;d)}
\def\BK{{{\mathbb K}}}
\def\BS{{{\mathbb S}}}

\def\Sym{{\rm{Sym}}}

\numberwithin{equation}{section} 
\newtheorem*{thm1}{Theorem}
\newtheorem*{prop1}{Proposition}
\newtheorem{thm}{Theorem}[section]
\newtheorem{prop}[thm]{Proposition}
\newtheorem{lem}[thm]{Lemma}

\theoremstyle{definition}

\newtheorem{cor}[thm]{Corollary}
\newtheorem{rmk}[thm]{Remark}
\newtheorem{ex}[thm]{Example}

\excludecomment{OmittedProofs}

\title{Asymptotic Schur decomposition of Veronese syzygy functors}
\author{Mihai Fulger, Xin Zhou}
\date{}

\begin{document}

\maketitle
\tableofcontents

\section*{Introduction}
The syzygies of the $d$-th Veronese embedding of $\mathbb P(V)$ are functors of the complex vector space $V$. We show that as $d$ grows, their Schur functor decomposition is very rich whenever they are not zero. This is deduced from an asymptotic study of related plethysms. We also obtain results related to the question \cite[Prob. 7.8]{EL}.
\par Turning to details, denote by $K_{p,q}(V;d)$ the vector space generated by the minimal generators in degree $(p+q)$ of the $p$-th module of syzygies of the section ring of $\O_{\P(V)}(d)$ as a $\Sym^{\bullet}H^0(\mathcal O_{\mathbb P(V)}(d))-$module. Equivalently, $K_{p,q}(V;d)$ is the cohomology of the Koszul-type complex 
\begin{equation*}
\bigwedge\nolimits^{p+1}\Sym^dV\otimes\Sym^{(q-1)d}V\to\bigwedge\nolimits^{p}\Sym^dV\otimes\Sym^{qd}V\to
\bigwedge\nolimits^{p-1}\Sym^dV\otimes\Sym^{(q+1)d}V.
\end{equation*}
A first subject of interest (cf. \cite{BCR}, \cite{BCR2}, \cite{EL}, \cite{G}, \cite{OP}, \cite{OR}, \cite{R}, \cite{Sch}) is understanding when these cohomology groups are zero or not. Nonvanishing results on $\mathbb P^2$ by Ottaviani and Paoletii \cite{OP} have been extended as part of their work in \cite{EL} by Ein and Lazarsfeld, and independently by J. Weyman. They show that for a fixed large $d$, essentially all the syzygies not covered by classical vanishing results (obtained from Castelnuovo--Mumford regularity considerations) are nonzero. The second author has found a simpler proof in \cite{Z}. 
\vskip.5cm

In this paper, we are interested in how the decompositions of syzygy modules behave asymptotically.\footnote{A non-equivariant study of growth can be found in \cite{EEL} where Ein, Erman, and Lazarsfeld study a "random" Betti table and provide evidence for their conjecture that Betti numbers become normally distributed as $d$ grows.} Specifically, the association $$V\to K_{p,q}(V;d)$$ is a functor that we denote $\BK_{p,q}(d)$. As in \cite{EL}, \cite{R}, \cite{Snow}, we can then write 
\begin{equation}\label{decomposition}
\BK_{p,q}(d)=\bigoplus_{\lambda\vdash(p+q)d}M_{\lambda}(p,q;d)\otimes_{\C}\BS_{\lambda},
\tag{$\dagger$}
\end{equation} 
where $M_{\lambda}(p,q;d)$ is a complex vector space whose dimension measures the multiplicity in $\BK_{p,q}(d)$ of the Schur functor $\mathbb S_{\lambda}$ corresponding to the partition $\lambda$ of $(p+q)d$. 

We aim to give asymptotic measures for the complexity of \eqref{decomposition}. Dictating our considerations are the following classical results:
\begin{enumerate}[a)]
\item A result of Green from \cite{G} implies that
$$\BK_{p,q}(d)=0,\mbox{ if }q\geq 2,\mbox{ and } d \geq p.$$
\item It is elementary that $$\BK_{p,0}(d)=0,\mbox{ if }p\geq 1,\ \BK_{0,1}(d)=0,\mbox{ and }\BK_{0,0}(d)=\mathbb C.$$
\end{enumerate}
Hence if $p$ and $q$ are fixed, and $d$ grows, then $\BK_{p,q}(d)$ can be nontrivial only when $p\geq1$ and $q=1$. Our main result (see also Theorem \ref{kp10d}) gives a measure for how in this case its Schur-functor decomposition is not just nontrivial, but very rich.

\begin{thm1}Fix $p\geq1$. Then as $d$ grows, $\BK_{p,1}(d)$ contains\footnote{We use $\Theta(f)$ to denote functions bounded below and above by multiples of $f$. Some more precise statements appear in the body of the paper. Compare ii) with Theorem \ref{kp10d}.ii).}
\begin{enumerate}[i)]
\item $\Theta(d^p)$ distinct Schur functors.
\item $\Theta(d^{p+1 \choose 2})$ Schur functors counting multiplicities.
\end{enumerate}
\end{thm1}

\vskip.5cm
\par For $q=0$, interesting asymptotic behavior is obtained when we introduce a new parameter. We define $K_{p,q}(V,b;d)$ to be the vector space generated by the minimal generators in degree $(p+q)$ of the $p$-th module of syzygies of the section ring of $\O_{\P(V)}(d)$ twisted by  $\O_{\P(V)}(b)$ as a $\Sym^{\bullet}H^0(\mathcal O_{\mathbb P(V)}(d))-$module. In terms of Koszul cohomology, it is computed by
$$\bigwedge\nolimits^{p+1}\Sym^dV\otimes\Sym^{(q-1)d+b}V\to\bigwedge\nolimits^{p}\Sym^dV\otimes\Sym^{qd+b}V\to
\bigwedge\nolimits^{p-1}\Sym^dV\otimes\Sym^{(q+1)d+b}V.$$ Historically, the introduction of the parameter $b$ was the key to creating a convenient formalism for inductive proofs, e.g., restriction to hypersurfaces. We define $\BK_{p,q}(b;d)$ as before. A version of Green's result shows that for fixed $p$, only $q=0$ and $q=1$ can provide nonzero functors when $d$ is large enough. In similar vein to the previous theorem, we prove a statement for $q = 0$:

\begin{thm1}Fix $p\geq 1$ and $b\geq1$. As $d$ grows, $\BK_{p,0}(b;d)$ contains:
\begin{enumerate}[i)]
\item $\Theta(d^{p-1})$ distinct Schur functors.
\item $\Theta(d^{p\choose2})$ Schur functors counting multiplicities.
\end{enumerate}
\end{thm1}

We also give a statement concerning the case $q = 1$ in the same setting:
\begin{prop1}
Fix $p\geq1$ and $b\geq1$. Assume that $p\geq b+1$. Then as $d$ grows, the logarithms of the number of distinct types of Schur subfunctors and of the number of Schur subfunctors counting multiplicities in $\BK_{p,1}(b;d)$ are $\Theta(\log d)$.
\end{prop1}

\par Varying other parameters is also of interest. In a multivariate setting, Raicu \cite{Ra} shows that as $b$ grows, the decompositions of the syzygies of Segre--Veronese embeddings stabilize. 
Fixing $d$ and allowing $p$ to grow, we prove:

\begin{prop1}Fix $b\geq1$ and $d\geq3$. Then as $p$ grows,
the logarithm of the number of distinct types of Schur subfunctors of $\BK_{p,0}(b;d)$ is $\Theta(\sqrt{p})$.
\end{prop1}

We also give a lower bound for the number of distinct Schur subfunctors in $\BK_{p,1}(b;d)$ increasing $d$ with $p$. (See Proposition \ref{increase d with p}.)

\vskip.5cm
\par
The idea of the proof of Theorem \ref{kp10d} is to show that the Schur decomposition of $\BK_{p,1}(d)$ is as rich as that of $\bigotimes^{p+1}\Sym^d$, in which each of the terms of the Koszul complex defining $\BK_{p,1}(d)$ embeds. We prove that the sum of all multiplicities from the left term in the Koszul complex is asymptotically $1/(p+1)!$ of the total multiplicity of $\bigotimes^{p+1}\Sym^d$, whereas for the middle term, the corresponding factor is $1/p!$, and the right term is asymptotically insignificant in comparison. These explain the assertion on the total multiplicity. The growth of the number of distinct types of functors in $\BK_{p,1}(d)$ is deduced by studying a moment map. 
\par The asymptotics of the Schur decomposition of $\bigotimes^{p+1}\Sym^d$ are deduced via Pieri's rule from a convex-geometric approach, flavors of which already exist in the literature (e.g., \cite{BZ}, \cite{K}, \cite{KK}, \cite{Kh}, \cite{O}). Then a result of Howe \cite{Ho} can be applied to study the asymptotic decompositions of each of the terms in the defining Koszul-type complex of $\BK_{p,1}(d)$.
\vskip.5cm
\par We carry these ideas out in the first two sections. In the first section we study the asymptotics in $d$ of the Schur decomposition of $\bigotimes^p\Sym^d$, then of $\mathbb S_{\mu}\Sym^d$, where $\mu$ is a fixed partition of $p$. In the second section we deduce our results on $\BK_{p,1}(d)$ and on $\BK_{p,0}(b;d)$ when $b>0$. In the third section we explain why new ideas are needed for the study of $\BK_{p,1}(b;d)$ when $b>0$. We then sketch a restriction argument inspired by \cite{EL}, which under additional assumptions leads to a weaker asymptotic description of the number of distinct types of Schur functors in the decomposition of $\BK_{p,1}(b;d)$ as $d$ grows. In the fourth section we study asymptotic behaviors of syzygies when $p$ grows.

\begin{paragraph}{Acknowledgments} We have benefited from useful discussions with Samuel Altschul, Igor Dolgachev, Daniel Erman, William Fulton, Roger Howe, Thomas Lam, and Mircea Musta\c t\u a. We thank Claudiu Raicu and David Speyer for providing many comments, suggestions, and improvements on a preliminary draft. We are grateful to our advisor Robert Lazarsfeld for posing the problem, and for numerous suggestions and encouragements. \end{paragraph} 

\addtocontents{toc}{\protect\setcounter{tocdepth}{0}}
\section{Notation and facts}\label{Facts from Representation Theory}
\addtocontents{toc}{\protect\setcounter{tocdepth}{2}}

\begin{paragraph}{1.}
We adopt the notation and definitions of \cite{FH} and \cite{F} for the basic objects of the Representation Theory of the general linear group. In particular, we write $\lambda\vdash n$ when $\lambda=(\lambda_1\geq\lambda_2\geq\ldots)$, with $\sum_{i\geq1}\lambda_i=n$, is a \textit{partition} of $n$. The \textit{length} of $\lambda$, denoted $|\lambda|$, is the number of its nonzero parts. We write $(2^2, 1)$ for the partition $(2,2, 1)$, we write $\lambda+\mu$ for the partition $(\lambda_1+\mu_1,\ \lambda_2+\mu_2,\ldots)$, and $2\lambda$ for the partition $(2\lambda_1,\ 2\lambda_2,\ldots)$, etc.  
\end{paragraph}

\begin{paragraph}{2.}
We write $S^d$ for $\Sym^d$. The symmetric group on $p$ elements is denoted by $\Sigma_p$.
\end{paragraph}

\begin{paragraph}{3.} In this paper we work with functors $\mathcal F:{\rm Vect}_{\mathbb C}\to{\rm Vect}_{\mathbb C}$ of finite dimensional complex vector spaces that have unique finite direct sum \textit{decompositions}\footnote{Plethysm functors are polynomial, and syzygy functors have decompositions by \cite[\S 2]{R}} $$\mathcal F=\bigoplus_{\lambda}M_{\mathcal F,\lambda}\otimes_{\mathbb C}\BS_{\lambda},$$ where $M_{\mathcal F,\lambda}$ are complex vector spaces, and $\BS_{\lambda}$ is the Schur functor corresponding to the partition $\lambda$. The dimensions
$$(\mathcal F,\lambda)=_{\rm def}\dim M_{\mathcal F,\lambda}$$ are the \textit{multiplicities} of $\BS_{\lambda}$ in $\mathcal F$. The \textit{total multiplicity} of $\mathcal F$ is $$N(\mathcal F)=_{\rm def}\sum_{\lambda}(\mathcal F,\lambda).$$ The \textit{complexity} of $\mathcal F$ is the number of distinct types of Schur functors appearing in its decomposition, i.e., $$c(\mathcal F)=_{\rm def}\#\{\lambda:\ (\mathcal F,\lambda)\neq 0\}.$$ If $V$ is a complex vector space of finite dimension, then $\mathcal F(V)$ is naturally a $GL(V)$-representation. We will often use this to reduce questions about functors to questions about representations of the general linear group.
\end{paragraph}

\begin{paragraph}{4.} It is an elementary consequence of Pieri's rule that $(\bigotimes^pS^d,\lambda)$ is equal to the number of semistandard Young tableaux of shape $\lambda$ and weight $\mu=(d^p)$.\footnote{Such tableaux can be described by Gelfand--Tsetlin patterns (see \cite[p133]{S} for a definition). The number of such Young tableaux is called the Kostka number $K_{\lambda\mu}$} 
In particular, the Schur functor $\BS_\lambda$ appears in $\bigotimes^p S^d$ if, and only if, $\lambda$ is a partition of $pd$ with at most $p$ parts. 
\begin{OmittedProofs}
\begin{shaded}
We claim that if $\lambda$ is a partition of $pd$ then $S_\lambda$ appears in $\bigotimes^p S^d$ if and only if $\lambda$ has at most $p$ parts. We see this as follows. By 4 above, $S_\lambda$ appears in $\bigotimes^p S^d$ if and only if $K_{\lambda \mu} \neq 0$. For general partitions $\lambda,\mu$, $K_{\lambda \mu} \neq 0$ if and only if they are partitions of the same integer and $\lambda$ is large than $\mu$ in dominance order. So any partition appearing in $\bigotimes^pS^d$ has to be a partition of $pd$ and it can have at most $p$ parts because of Pieri's rule. Conversely, for $1 \leq k \leq p$, the sum of the $k$ largest parts of $(d,\dots, d)$ is $kd$. If for $\lambda$, $\lambda_1+ \dots + \lambda_k < kd$, then $\lambda_{k+1} + \dots + \lambda_p > (p-k)d$, then $\lambda_k < d <  \lambda_{k+1}$, contradiction to the fact that $\lambda$ is a partition. So $\lambda_1+ \dots + \lambda_k < kd$ is impossible and therefore, by definition, $\lambda$ is larger than $\mu$ in dominance order.  
\end{shaded}
\end{OmittedProofs}
\end{paragraph}

\begin{paragraph}{5.} Let $V$ be a $GL_n$-representation. Denoting by $U_n$ the unipotent group, by \cite[p144]{F}, the total multiplicity of $V$ is equal to $\dim V^{U_n}$ (the space of $U_n$-invariant vectors in $V$). It is classical that $\mathbb S_{\lambda}(\C^p)\neq 0$ if and only if, $|\lambda|\leq p$. If $\mathcal F:{\rm Vect}_{\C}\to{\rm Vect}_{\C}$ is a decomposable functor (as in Fact 3), such that any Schur subfunctor corresponds to a partition with at most $p$ rows, then the Schur decomposition of $\mathcal F$ can be read from the decomposition into irreducible $GL_p$-subrepresentations of $\mathcal F(\C^p)$. In particular, we can read the total multiplicity and the complexity of $\mathcal F$ by applying the functor to $\C^p$, i.e., $N(\mathcal F)=\dim \mathcal F(\C^p)^{U_p}.$
\end{paragraph}

\section{Asymptotic plethysms}

In this section we investigate the growth with $d$ of the total multiplicity, and of the complexity of $\bigotimes^p\Sym^d$, $\Sym^p\Sym^d$, and $\bigwedge^p\Sym^d$ when $p$ is fixed. As in several of the references listed in the remark below, the idea behind our study is convex-geometric. The contents of this section may be known to the experts, but we were unable to find precise references. We give a presentation here for the convenience of the reader. 

\begin{rmk}\label{plethysm refs}\textnormal{
From the literature on the asymptotics of the decomposition of $\bigotimes^pS^d$ and other plethysms, we mention the following:
\begin{itemize}
\item The work of Kaveh and Khovanskii in \cite{KK} applies to the behavior of $\bigotimes^pS^d$ with fixed $p$ and varying $d$. Their focus is on subtle properties such as Fujita-type approximation, the Brunn--Minkowski inequality, the Brion--Kazarnowskii formula, etc. 
\item In \cite{K}, Kaveh computes the dimension of the moment body, varying $p$, for $\bigotimes^pV$, i.e., the growth with $p$ of the number of its distinct subrepresentations, where $V$ is a fixed $GL_n$-representation. 
\item Tate and Zelditch (\cite{TZ}) studied the asympototics of Kostka numbers of $\bigotimes^pS_\lambda$ where $p$ varies and $S_\lambda$ is a fixed representation. 
\item Different asymptotic plethysms have been studied in \cite{Wei}. In \cite{Man}, we again find the idea of the convex-geometric approach. 
\end{itemize}
}
\end{rmk}

\subsection{Integral points and $\bigotimes^pS^d$}
In this subsection, we show that the total multiplicity and the complexity of $\bigotimes^pS^d$ are counted by the number of lattice points inside slices over $d$ of two rational convex cones. The growth with $d$ of the number of such integral points is a polynomial of degree equal to the dimension of the cross section of the corresponding cone. We determine these two cones and compute the corresponding dimensions.\footnote{In the language of \cite{KK},  we are determining the dimension of the moment body and of the multiplicity body (which in our case is also the classical Gelfand--Tsetlin polytope) for the $p$-th product of a sufficiently large dimension projective space.} These are captured by the following theorem:
\begin{thm}\label{tensorsymthm}Fix $p\geq 1$. Then 
\begin{enumerate}[i)]
\item $\lim_{d\to\infty}c(\bigotimes^pS^d)/d^{p-1}$ is a finite positive number.
\item $\lim_{d\to\infty}N(\bigotimes^pS^d)/d^{p\choose 2}$ is a finite positive number.
\end{enumerate}
\end{thm}

\begin{proof} We specify two graded sets $Y_\bullet$ and $\mathcal{Y}_\bullet$ such that the cardinalities of their $d$-th graded pieces count the complexity and the total multiplicity of $\bigotimes^pS^{\bullet}$, respectively. 

\begin{paragraph}{Complexity.}The set
\begin{equation*}
 Y_d =_{\rm def}\left\{(\lambda_2,...\lambda_p,d)\ :\  \lambda_1 \geq \lambda_2 \geq ... \geq \lambda_p \geq 0,\mbox{ with }  \lambda_1:= pd-\sum_{k=2}^p\lambda_k \right\}\subseteq \mathbb{N}^{p-1} \times \{d\}
\end{equation*}
is a parametrization of the partitions $\lambda\vdash pd$ with at most $p$ parts, which is the set of distinct types of Schur functors appearing in the decomposition of $\bigotimes^pS^d$ by Fact 4 from \S 0. Therefore the complexity of $\bigotimes^p S^d$ is the number of integral points with the last coordinate $d$ inside the cone contained in $\mathbb R_{\geq0}^{p-1}\times\mathbb R_{\geq0}$ over the cross section:
$$Y =_{\rm def}\left\{(\lambda_2,...\lambda_p,1)\ :\  \lambda_1 \geq \lambda_2 \geq ... \geq \lambda_p \geq 0,\mbox{ with } \lambda_1 := p-\sum_{k=2}^p \lambda_k \right\}\subset \mathbb{R}_{\geq 0}^{p-1} \times \{1\}.$$
This is $(p-1)$-dimensional, which proves part $i)$.
\end{paragraph}

\begin{paragraph}{Total multiplicity.}Guided by Fact 4, we seek to parameterize the set of semistandard Young tableaux $T$ with $pd$ boxes, at most $p$ rows, and weight $(d^p)$. A first parameterization is the $p\times p$ matrix 
\begin{center}$t_{ij} =_{\rm def}$ number of $j$'s in the $i$-th row of $T$.
\end{center}

\noindent The semistandard tableau condition implies $t_{ij}=0$ when $j<i$. The restriction on the weight of $T$ imposes $$t_{jj} = d- \Sigma_{k=1}^{j-1}t_{kj}\mbox{ for all }j\in\{1,\ldots,p\}.$$ Hence $T$ is determined by $(t_{ij})_{i<j}$ and by $d$. With this parameterization, the set of all such $T$ is in bijection with the set $\mathcal Y_d$ of of points in $\mathbb N^{p\choose 2}\times\{d\}=\{((t_{ij})_{i<j},d)\}$ subject to the following conditions:

\begin{itemize}
\item Denoting $t_{ii}:=d-\sum_{j=1}^{i-1}t_{ji},$ we ask that \begin{equation}t_{ii}\geq 0,\end{equation} i.e., the number of $j$'s on the $i$-th row should be nonnegative even when $j=i$. 
\item We also ask that 
\begin{equation}\sum_{k=i}^{j-1}t_{ik} \geq \sum_{k=i+1}^j t_{i+1,k}, \ \mbox{for all }1 \leq i < p,\ 1 \leq j \leq p.\end{equation}
This is the tableau condition: we ask that labels greater than or equal to $j$ may start to appear on the $(i+1)$-st row of $T$ only on columns to the left or below where labels smaller than or equal to $j-1$ stopped on the $i$-th row. (Note that applying this with $j=p$ implies that the shape of $T$ is a Young diagram.)
\end{itemize}

\noindent Similar to the complexity problem, $\mathcal Y_d$ is the set of integral points inside the cone with the last coordinate $d$ in $\mathbb R_{\geq0}^{p\choose2}\times\mathbb R_{\geq0}$ over the cross section:
$$\mathcal Y:=\left\{((t_{ij})_{i<j},1)\ :\ \begin{array}{c}
0\leq t_{ii}:=1-\sum_{k=1}^{i-1}t_{ki},\mbox{ for all }1\leq i\leq p\\
\sum_{k=i}^{j-1}t_{ik} \geq \sum_{k=i+1}^j t_{i+1,k}, \ \mbox{for all }1 \leq i < p,\ 1 \leq j \leq p\end{array}\right\}\subset\mathbb R_{\geq0}^{p\choose 2}\times\{1\}.$$
This convex body is $p \choose 2$-dimensional,\footnote{To see this, it is enough to produce a point of $\mathcal Y$ that satisfies strictly all the defining inequalities. Setting $t_{ij}=p^i \epsilon$ for all $i<j$, and for sufficiently small $\epsilon$, defines such a point.} and part $ii)$ follows.
\end{paragraph}
\end{proof}

\begin{rmk} The limits in Theorem \ref{tensorsymthm} are at least algorithmically computable for each $p$, since they are the volumes of the convex bodies $Y$ and $\mathcal Y$ respectively.\end{rmk}

\begin{ex}[The plethysm of $\bigotimes^2 S^d$]\textnormal{
The decomposition of $\bigotimes^2 S^d$ consists of all Schur functors of type $(2d-a,a)$ with $0 \leq a \leq d$, each with multiplicity $1$. The total multiplicity and the complexity are both $d$ in this case. \qed
}
\end{ex}

\begin{rmk}\textnormal{
Our constructions are similar to those of the Gelfand--Tsetlin patterns (see also \cite{BZ}). We use this particular form in order to compute the wanted growth rates. } \qed
\end{rmk}

The following remark is a consequence of the proof of Theorem \ref{tensorsymthm}. It allows us to deem certain collections of Schur functors asymptotically insignificant and will be used repeatedly:
\begin{rmk}\label{few}The coordinates corresponding to Young tableaux of weight $(d^p)$ with at most $p-1$ rows, and to Young diagrams with $pd$ boxes and at most $p-1$ rows lie in proper linear subspaces of the real vector spaces spanned by $\mathcal Y_{\bullet}$ and $Y_{\bullet}$ respectively. It follows that asymptotically all the Young tableaux and diagrams parameterized by $\mathcal Y_{\bullet}$ and $Y_{\bullet}$ respectively have $p$ rows.\qed\end{rmk}

\subsection{$\mathbb S_{\mu}S^d$ via $\bigotimes\nolimits^pS^d$}

We investigate the asymptotics of the Schur decomposition of $\mathbb S_{\mu}S^d$ when $\mu$ is a fixed partition of $p$. Denote by $V_{\mu}$ an irreducible complex $\Sigma_p$-representation of weight $\mu$.

\begin{thm}\label{t2p}Fix $p\geq 1$, and let $\mu$ be a partition of $p$. Then 
\begin{enumerate}[i)]
\item $\lim_{d\to\infty}N(\mathbb S_{\mu}S^d)/d^{p\choose 2}=\frac{\dim V_{\mu}}{p!}\cdot\lim_{d\to\infty}N(\bigotimes^pS^d)/d^{p\choose2}$.
\item As $d$ grows, $c(\mathbb S_{\mu}S^d)=\Theta(d^{p-1})$.
\end{enumerate}
\end{thm}
\begin{proof}
We have inclusions of graded (by $d)$ vector spaces $A,B,C$ (whose graded pieces are denoted $A_d, B_d, C_d$ respectively):
$$A:=\bigoplus_{d \geq 0} (V_{\mu}\otimes\mathbb S_{\mu}S^d\C^p)^{U_p} \hookrightarrow B:=\bigoplus_{d \geq 0}(\bigotimes\nolimits^p S^d\C^p)^{U_p} \hookrightarrow C:=\bigoplus_{d \geq 0}\bigotimes\nolimits^pS^d\C^p.$$ 
The first inclusion is induced by the functorial decomposition $$\bigotimes^p=\bigoplus_{\mu\vdash p}V_{\mu}\otimes\mathbb S_{\mu}.$$ Observe that $C$ is the section ring on the $p$-fold product of $\P(\C^p)$ of $\O_{\P(\C^p)^{\times p}}(1,\ldots,1)$, and $B$ also has an algebra structure. The action of $GL_p$ on $\C^p$ induces a $GL_p$-representation structure on $C_d$ for all $d$. The symmetric group $\Sigma_p$ also acts on $C_d=\bigotimes^pS^d\C^p$ by permuting the factors of the tensor product. The two actions commute, hence the $\Sigma_p$-action restricts to $B$. We have the following:
\begin{enumerate}
\item By \cite[Thm. 16.2]{Gr}, $B$ is finitely generated. 
\item No nontrivial $\sigma$ in $\Sigma_p$ acts as a scalar (depending only on $\sigma$ and $d$) on $B_d$ for each $d$. Otherwise, by the commutativity of the actions of $\Sigma_p$ and $GL_p$, it also acts as a scalar on the $GL_p$-span of $B_d$ in $C_d$, which is the entire $C_d$ by Fact 5. 
\end{enumerate}
Using \cite{Ho} (see \cite{P} for a geometric perspective) and $\S 0$ Fact 5, these give:
$$\dim A_d\sim\frac{(\dim V_{\mu})^2}{|\Sigma_p|}\cdot\dim B_d.$$
By Theorem \ref{tensorsymthm} we have part $i)$. Part $ii)$ is a consequence of $i)$ and Lemma \ref{maxgrowthsubts}.\end{proof} 

\begin{lem}\label{maxgrowthsubts}Fix $p\geq 1$, and let $\mathcal F_d$ be a sequence of subfunctors of $\bigotimes^pS^d$ so that, as $d$ grows, we have $N(\mathcal F_d) \in \Theta(d^{p\choose 2})$.  Then $c(\mathcal F_d) \in \Theta(d^{p-1})$.\end{lem}
\begin{proof}Consider the moment map $$\mu:\mathbb R^{p\choose 2}\times\mathbb R\to \mathbb R^{p-1}\times\mathbb R,\quad \mu((t_{ij})_{i<j},d):=(\lambda_2,\ldots,\lambda_p,d),\mbox{ where}$$
$$\lambda_i:=\sum_{j=i}^pt_{ij},\mbox{ for all }1\leq i\leq p,\mbox{ and}\quad t_{ii}:=d-\sum_{k=1}^{i-1}t_{ki},\mbox{ for all }1\leq i\leq p.$$
By the constructions of the previous section, $\mu$ maps $\mathcal Y$ onto $Y$, and $\mathcal Y_d$ onto $Y_d$ for all $d$. By assumption, there exists $C_1>0$ such that for large $d$, $$N(\mathcal F_d)\geq C_1 \cdot d^{p\choose 2}.$$ On the other hand, denoting by $\overline m_{p,d}$ the maximal multiplicity of a Schur functor in $\bigotimes^pS^d$, which contains $\mathcal F_d$, we have
$$N(\mathcal F_d)\leq c(\mathcal F_d)\cdot \overline m_{p,d}.$$
By Theorem \ref{tensorsymthm} i) we obtain $c(\mathcal F_d)\leq c(\bigotimes^pS^d) \leq C_2 \cdotp (d^{p-1})$, it is enough to show that $$\overline m_{p,d}\leq C'\cdot d^{{p-1}\choose 2}$$ for some $C'>0$ independent of $d$. To see this, choose a basis for $\mathbb Z^{p\choose 2}\subset\mathbb R^{p\choose2}$ so that $\mu$ is the projection onto the last $p$ coordinates. Then choose an integer $l>0$ such that $$\mathcal Y\subset[-l,l]^{{p-1}\choose 2}\times Y\subset\mathbb R^{p\choose2}\times\{1\}.$$ Then $\mathcal Y_d\subset [-dl,dl]^{{p-1}\choose2}\times Y_d$, so $\overline m_{p,d}\leq \#([-dl,dl]\cap\mathbb Z)^{{p-1}\choose2}$. We can set $C'=(3l)^{{p-1}\choose2}$.\end{proof}

\begin{cor}\label{t2sw}
Fix $p\geq 1$. Then 
$$\frac1{p!}N(\bigotimes\nolimits^pS^d) \sim N(S^pS^d) \sim N(\bigwedge\nolimits^pS^d).$$
\end{cor}
\begin{proof}
This follows by applying Theorem \ref{t2p} for the trivial and alternating representations of $\Sigma_p$.
\end{proof}

\begin{rmk}\label{t2sc}
 \cite[Lem. 2.2]{BCI} can be adjusted to show that if $\mathbb S_{\lambda}(\C^p)$ and $\mathbb S_{\lambda'}(\C^p)$ appear in the decompositions of $S^pS^d\C^p$ and $S^pS^{d'}\C^p$ respectively, then $\mathbb S_{\lambda+\lambda'}(\C^p)$ appears in the decomposition of $S^pS^{d+d'}\C^p$.\footnote{The product of two highest weight symmetric vectors is a highest weight symmetric vector} In particular, the set of $\lambda$ with $(S^pS^{d},\lambda)>0$, for some $d$, is a subsemigroup of $Y_{\bullet}$. Since $S^pS^d\neq 0$ for all $d\geq0$, the semigroup is nonempty in all degrees. The existence and finiteness of  $$\lim_{d\to\infty}c(S^pS^d)/d^{p-1}$$ are easy applications of the semigroup techniques developed in \cite{Kh} and \cite[\S 2.1]{LM}. Also,
$$c(\bigwedge\nolimits^pS^d)\sim c(S^pS^d).$$ We leave it to the reader to deduce this from Remark \ref{few} and from the following:
\end{rmk}

\begin{lem}[Newell]\label{Newell}For any partition $\lambda=(\lambda_1\geq\ldots\geq\lambda_p\geq0)$, we have
\begin{enumerate}[i)]
\item $(\bigwedge^pS^{d+1},(\lambda_1+1,\ldots,\lambda_p+1))=(S^pS^d,\lambda)$.
\item $(S^pS^{d+1},(\lambda_1+1,\ldots,\lambda_p+1))=(\bigwedge^pS^d,\lambda)$.
\end{enumerate}
\end{lem}

\begin{rmk}[Explicit constructions]\label{wexplicit} Answering a question from \cite[Conj. 2.11]{Wei}, in \cite{BCI} it is shown that if $\lambda\vdash pd$ has $|\lambda|\leq p$, then $$\mathbb S_{2\lambda}\mbox{ appears in the decomposition of }S^pS^{2d}\mbox{, i.e., }(S^pS^{2d},2\lambda)>0.$$ In \cite{MM} we find a constructive proof for this result. 
Because of the semigroup structure on types of partitions contributing to $S^pS^{\bullet}$, and using Newell's result, it also holds that
$$(\bigwedge^p S^{2d+1},2\lambda+(1^p))=(S^pS^{2d},2\lambda)>0,\quad\mbox{and}$$
$$(\bigwedge^p S^{2d+2},2\lambda+(1^p)+(p))=(S^pS^{2d+1},2\lambda+(p))>0.$$ 
Combining with Theorem \ref{tensorsymthm}.ii), these provide a constructive proof for Theorem \ref{t2p}.ii) in the particular cases when $\mu=(p)$ or $\mu=(1^p)$, but do not show that the corresponding limits exist, as it is the case in Remark \ref{t2sc}. The explicit constructions also show that $\lim_{d\to\infty}\frac{c(S^pS^d)}{d^{p-1}}\geq\frac1{2^{p-1}}\cdot\lim_{d\to\infty}\frac{c(\bigotimes^pS^d)}{d^{p-1}}$. Note that $\frac1{2^{p-1}}>\frac1{p!}$, when $p\geq3$. In particular, it is not true in general that $c(S^pS^d)\sim\frac1{p!}c(\bigotimes^pS^d)$.
\end{rmk}

In the next section we will also use the following result:
\begin{prop}\label{asy1}Fix $p\geq1$. Then, as $d$ grows, $N(\bigwedge^pS^d\otimes S^d)\sim\frac1{p!}\cdot N(\bigotimes^{p+1}S^d).$\end{prop}
\begin{proof}This is the same argument as for Corollary \ref{t2sw}. Apply Theorem \ref{t2p}.i) to the alternating representation of $\Sigma_p$ with the action on the first $p$ tensor factors of $\bigotimes^{p+1}S^d$.\end{proof}

\section{Asymptotic syzygy functors}
Given $p,q,b,d\geq 0$, we defined the syzygy functor $\Bpq$ of the $d$-th Veronese embedding as the cohomology of the functorial Koszul-type complex
$$\bigwedge\nolimits^{p+1}S^d\otimes S^{(q-1)d+b}\to\bigwedge\nolimits^{p}S^d\otimes S^{qd+b}\to\bigwedge\nolimits^{p-1}S^d\otimes S^{(q+1)d+b}.$$
In the introduction we explained that if we fix $p,q,b\geq 0$ and let $d$ grow to infinity, we only have nontrivial behavior in $\Bpq$ when $p\geq 1$, and either $q=1$, or $q=0$ and $b\geq 1$. In this section we determine the precise asymptotic orders of growth for the total multiplicity and complexity of $\BK_{p,1}(0;d)$ and $\BK_{p,0}(b;d)$ in the nontrivial case $b\geq1$. In the next section we will give a partial result for $\BK_{p,1}(b;d)$ when $b\geq 1$.

\subsection{$\mathbb{K}_{p,1}(d)$}
As we did in the introduction, we write $\BK_{p,1}(d)$ for $\BK_{p,1}(0;d)$, i.e., the cohomology of the Koszul-type complex
\begin{equation}\label{kpqeq}\bigwedge\nolimits^{p+1}S^d\to \bigwedge\nolimits^pS^d\otimes S^d\to\bigwedge\nolimits^{p-1}S^d\otimes S^{2d}.\end{equation}
The asymptotics of the decomposition of $\BK_{p,1}(d)$ are described by:
\begin{thm}\label{kp10d}Fix $p\geq 1$. As $d$ goes to infinity:
\begin{enumerate}[i)]
\item $N(\BK_{p,1}(d)) \sim \frac{p}{(p+1)!}\cdotp N(\bigotimes\nolimits^{p+1}S^d)$.
\item $c(\BK_{p,1}(d))\in \Theta(d^p).$
\end{enumerate}
\end{thm}
\begin{proof} Since the first map in \eqref{kpqeq} is an inclusion, and by Remark \ref{few} the total multiplicity of the last term is asymptotically insignificant compared to that of $\bigotimes^{p+1}S^d$, we obtain that $$N(\BK_{p,1}(d))\sim N(\bigwedge\nolimits^pS^d\otimes S^d)-N(\bigwedge\nolimits^{p+1}S^d)$$ once we have proved that the right hand side is in $\Theta(N(\bigotimes^{p+1}S^d))$ Using Corollary \ref{t2sw} and Proposition \ref{asy1},
$$N(\bigwedge\nolimits^pS^d \otimes S^d) - N(\bigwedge\nolimits^{p+1}S^d) =  \frac{1}{p!}\cdotp N(\bigotimes\nolimits^{p+1}S^d) - \frac1{(p+1)!}\cdotp N(\bigotimes\nolimits^{p+1}S^d),$$
which is in $\Theta(N(\bigotimes\nolimits^{p+1}S^d))$ and part $i)$ follows. Part $ii)$ is a consequence of Lemma \ref{maxgrowthsubts} after observing that $\BK_{p,1}(d)$ is noncanonically a subfunctor of $\bigotimes^{p+1}S^d$.\end{proof}

\begin{rmk}Unlike with $\bigotimes^{p+1}S^d$ and the plethysms of Remark \ref{wexplicit}, we do not know how to construct explicit examples of $\lambda$ with $(\BK_{p,1}(d),\lambda)>0$, nor do we know if the sequence $c(\BK_{p,1}(d))/d^p$ has a limit. \end{rmk}

\subsection{$\BK_{p,0}(b;d)$}

We describe the asymptotic behavior of 
\begin{equation}\label{kp0eq}\BK_{p,0}(b;d)=\ker(\bigwedge\nolimits^pS^d\otimes S^b\to\bigwedge\nolimits^{p-1}S^d\otimes S^{b+d}),\end{equation} in the the nontrivial cases when only $d$ grows, in the following theorem:

\begin{thm}\label{kp0bd when d grows}Fix $p\geq1$ and $b\geq1$. Then as $d$ grows,
\begin{enumerate}[i)]
\item $N(\BK_{p,0}(b;d))\in \Theta \left(d^{p\choose 2}\right).$
\item $c(\BK_{p,0}(b;d))\in \Theta (d^{p-1}).$
\end{enumerate}
\end{thm}
\begin{proof} Note that Pieri's rule implies $$(\bigwedge\nolimits^{p-1}S^d\otimes S^{b+d},\lambda)=0,\mbox{ for any }\lambda\mbox{ with }|\lambda|=p+1.$$ As in Remark \ref{few}, because $b$ is finite, asymptotically all of the Schur functors appearing in $\bigwedge^pS^d$ correspond to $\lambda$ with $\lambda_p\geq b$. For each occurrence of such $\mathbb S_{\lambda}$, by Pieri's rule, we have an occurrence of $\mathbb S_{(\lambda_1,\ldots,\lambda_p,b)}$ in $\bigwedge^pS^d\otimes S^b$. Since $b>0$, this corresponds to a partition with $p+1$ parts, hence it also appears in the decomposition of $\BK_{p,0}(b;d)$ by Schur's lemma and \eqref{kp0eq}. Therefore the total multiplicity and complexity of $\BK_{p,0}(b;d)$ are bounded below by those of $\bigwedge^pS^d$, which are described in Corollary \ref{t2sw} and Remark \ref{t2sc} respectively.
\par The total multiplicity and complexity of $\BK_{p,0}(b;d)$ are bounded above by those of $\bigotimes^pS^d\otimes S^b$. By Pieri's rule, there is at most a finite number $C(p,b)$ of ways of obtaining a Young tableau of weight $(d^p,b)$ by adding $b$ boxes labeled $p+1$ to a fixed one of weight $(d^p)$, or of ways of obtaining a partition $\lambda\vdash pd+b$ by adding $b$ boxes to the corresponding Young diagram of a fixed partition $\mu$ with $|\mu|\leq p$. Therefore \begin{equation}\label{counttc}N(\bigotimes\nolimits^p S^d\otimes S^b)\in \Theta (N(\bigotimes\nolimits^p S^d))\quad\mbox{and}\quad c(\bigotimes\nolimits^p S^d\otimes S^b)\in \Theta (c(\bigotimes\nolimits^p S^d)).\end{equation}\end{proof}

\begin{rmk}We have used the assumption $b\geq 1$ to say that adding $b$ boxes on the $(p+1)$-st row of a Young diagram with $p$ non-empty rows produces a digram with $p+1$ rows. The assumption is also necessary for the theorem because, as in the introduction, $\BK_{p,0}(0;d)=0$.\end{rmk}

\begin{rmk} Except for the case of the total multiplicity when $b=1$, which we will explain in the next section, we do not know if the sequences $N(\BK_{p,0}(b;d))/d^{p\choose2}$ and $c(\BK_{p,0}(b;d))/d^{p-1}$ have limits when $b\geq1$.\end{rmk}

\begin{lem}\label{remtctb} One can improve \eqref{counttc} to 
$$N(\bigotimes\nolimits^pS^d\otimes S^b)\sim{{b+p}\choose{p}}\cdot N(\bigotimes\nolimits^pS^d)\quad\mbox{and}\quad
c(\bigotimes\nolimits^pS^d\otimes S^b)\sim (b+1)\cdot c(\bigotimes\nolimits^pS^d).$$\end{lem}
\begin{proof}
By \S 0 Fact 4, the total multiplicity of $\bigotimes^pS^d\otimes S^b$ counts Young tableaux of weight $(d^p,b)$. These are partitioned according to their truncation to tableaux of weight $(d^p)$, by forgetting the $b$ boxes labeled $p+1$. Conversely, from any Young tableau $T$ of weight $(d^p)$ we obtain potential tableaux of weight $(d^p,b)$ by 
arbitrarily placing a total number of $b$ boxes labeled $p+1$ at the end of the first $p+1$ rows of $T$. The potential tableaux may fail to be actual tableaux only when the shape $\lambda$ of $T$ satisfies $\lambda_i<\lambda_{i+1}+b$ for some $i\leq p$. As in Remark \ref{few}, this phenomenon is asymptotically insignificant. Consequently, $$N(\bigotimes\nolimits^pS^d\otimes S^b)\sim {{b+p}\choose{p}}\cdot N(\bigotimes\nolimits^pS^d).$$ For the complexity problem, note that from a partition $\lambda\vdash pd$ with $|\lambda|\leq p$, one obtains $b+1$ potential partitions $\lambda[j]:=(\lambda_1+(b-j),\lambda_2,\ldots,\lambda_p,j)$ with $j\in\{0,\ldots,b\}$. By Pieri's rule, the set of all such $\lambda[j]$ that are true partitions (i.e., with $\lambda_p\geq j$) is the set of partitions $\mu$ with $(\bigotimes\nolimits^pS^d\otimes S^b,\mu)>0$.  Reasoning as in Remark \ref{few}, asymptotically all $\lambda[j]$ are true partitions. 
\end{proof}

\section{$\BK_{p,1}(b;d)$ when $b>0$}

In this section we investigate the asymptotics in $d$ of the decomposition of $\BK_{p,1}(b;d)$, when $b\geq 1$. We explain why the strategy from the previous section fails, and, using a restriction argument, we determine the asymptotic orders of the logarithms of the complexity and total multiplicity of $\BK_{p,1}(b;d)$ as functions of $\log d$, when $p$ and $b>0$ are fixed, satisfying $p\geq b+1$. \par The functor $\BK_{p,1}(b;d)$ is the cohomology of the functorial Koszul-type complex:
$$\bigwedge\nolimits^{p+1}S^d\otimes S^{b}\to\bigwedge\nolimits^pS^d\otimes S^{b+d}\to\bigwedge\nolimits^{p-1}S^d\otimes S^{b+2d}.$$
Similar to Lemma \ref{remtctb}, one can show 
$$N(\bigwedge\nolimits^{p+1}S^d\otimes S^b)\sim\frac{{{b+p+1}\choose{p+1}}}{(p+1)!}\cdot N(\bigotimes\nolimits^{p+1}S^d),$$
asymptotically $\frac{{{b+p}\choose p}}{(p+1)!}\cdot N(\bigotimes\nolimits^{p+1}S^d)$ of which corresponds to partitions of length $p+1$,
$$N(\bigwedge\nolimits^{p}S^d\otimes S^{b+d})\sim\frac 1{p!}\cdot N(\bigotimes\nolimits^{p+1}S^d),$$ and because the last term contains no Schur functor corresponding to partitions of length $p+1$,
$$N(\bigwedge\nolimits^{p-1}S^d\otimes S^{b+2d})\in o(N(\bigotimes\nolimits^{p+1}S^d)).$$
When $b>0$, the sum of multiplicities corresponding to partitions of length $p+1$ of the leftmost term is asymptotically at least as big as that of the central term. Hence we cannot apply the strategy of the previous section. In fact, with the help of a result of Raicu, when $b=1$ we obtain:  

\begin{prop}If $p\geq1$, then as $d$ grows, $N(\BK_{p,1}(1;d))\in o(N(\bigotimes^{p+1}S^d))$.\end{prop}
\begin{proof}Consider the sequence
$$\bigwedge\nolimits^{p+1}S^d\otimes S^1\to\bigwedge\nolimits^pS^d\otimes S^{d+1}\to \bigwedge\nolimits^{p-1}S^d\otimes S^{2d+1}.$$ The first two cohomologies are by definition $\BK_{p+1,0}(1;d)$ and $\BK_{p,1}(1;d)$. The total multiplicity of the third term is asymptotically insignificant compared to the first two. Hence $$N(\BK_{p+1,0}(1;d))-N(\BK_{p,1}(1;d))\sim N(\bigwedge\nolimits^{p+1}S^d\otimes S^1)-N(\bigwedge\nolimits^pS^d\otimes S^{d+1}),$$ if we show that the latter difference is in $\Theta (N(\bigotimes^{p+1}S^d))$. By the computations above,
$$N(\bigwedge\nolimits^{p+1}S^d\otimes S^1)-N(\bigwedge\nolimits^pS^d\otimes S^{d+1})\sim\frac1{(p+1)!}\cdotp N(\bigotimes\nolimits^{p+1}S^d).$$
By \cite[Thm. 6.4]{Ra}, the decomposition of $\BK_{p+1,0}(1;d)$ is obtained from the decomposition of $S^{p+1}S^{d-1}$ by replacing Schur subfunctors $\mathbb S_{\lambda}$ of the latter with $\mathbb S_{\lambda+(1^{p+2})}$.\footnote{By Pieri's rule and Lemma \ref{Newell}, this is the same as stating that $(\BK_{p+1,0}(1;d),\lambda)=0$ if $|\lambda|<p+2$. For this, as in \S 0 Fact 5, it is enough to check that $K_{p+1,0}(\C^{p+1},1;d)=0$, which follows from \cite[Prop. 5.1]{EL}.} In particular, $$N(\BK_{p+1,0}(1;d))=N(S^{p+1}S^{d-1})\sim \frac1{(p+1)!}\cdotp N(\bigotimes\nolimits^{p+1}S^d).$$
But this is the same approximation as for $N(\BK_{p+1,0}(1;d))-N(\BK_{p,1}(1;d))$.
\end{proof}

Even if the Schur decomposition of $\BK_{p,1}(b;d)$ when $b\geq 1$ may not be as rich as that of $\bigotimes^{p+1}S^d\otimes S^b$, we are able to give a weaker asymptotic description in Theorem \ref{kppositiveb}, when $p\geq b+1$. We use a restriction argument suggested by R. Lazarsfeld. The Schur subfunctors of $\BK_{p,q}(b;d)$ corresponding to partitions of length at most $n$ are captured by the $GL_{n}$-decomposition of $K_{p,q}(\C^{n},b;d):=\BK_{p,q}(b;d)(\C^{n}).$ By \cite[Prop. 3.2]{EL}, 
\begin{equation}\label{greenk}\mbox{If }q>0,\mbox{ then}\quad K_{p,q}(\C^n,b;d) = H^1\left(\bigwedge\nolimits^{p+1}M_d \otimes \O_{\P^{n-1}}\left(b+(q-1)d\right)\right),\mbox{ where}\end{equation}
$$M_d:=\ker(S^d(\C^n)\otimes\mathcal O_{\P^{n-1}}\to\mathcal O_{\P^{n-1}}(d)).$$
Consider a splitting $\C^n=\C^{n-1}\oplus\C$, and the $GL_{n-1}$-equivariant short exact sequence:
$$ \bigwedge\nolimits^{p+1}M_d \otimes \O_{\P^{n-1}}(b+(q-1)d) \hookrightarrow \bigwedge\nolimits^{p+1}M_d \otimes \O_{\P^{n-1}}(b+1+(q-1)d) \twoheadrightarrow \bigwedge\nolimits^{p+1}M_d \otimes \O_{\P^{n-2}}(b+1+(q-1)d).$$
Then from the long sequence in cohomology we extract the $GL_{n-1}$-equivariant complex:
$$H^0(\bigwedge\nolimits^{p+1}M_d \otimes \O_{\P^{n-1}}(b+1+(q-1)d)) \ra $$ \begin{equation}\label{cohs}\to H^0(\bigwedge\nolimits^{p+1}M_d \otimes \O_{\P^{n-2}}(b+1+(q-1)d)) \ra H^1(\bigwedge\nolimits^{p+1}M_d \otimes \O_{\P^{n-1}}(b+(q-1)d)).\end{equation}
Assume $n\geq2$ is chosen such that \begin{equation}\label{greenn}p+1 \geq h^0(\O_{\P^{n-1}}(b+1+(q-1)d)).\end{equation} By \cite[Thm. 3.a.1]{G}, the leftmost term of \eqref{cohs} is zero, therefore the rightmost map is an inclusion of $GL_{n-1}$-representations. Next, we study the $GL_{n-1}$-representation $$H^0(\bigwedge\nolimits^{p+1}M_d \otimes \O_{\P^{n-2}}(b+1+(q-1)d)).$$
By restricting the map defining $M_d$, and by comparing with the corresponding map for $\P^{n-2}$,
\begin{equation}\label{isorest}M_d\otimes\mathcal O_{\P^{n-2}}\simeq M'_d\oplus(\bigoplus_{j \leq d-1}S^j\C^{n-1}\otimes\mathcal O_{\P^{n-2}}),\end{equation} where $M'_d=\ker(S^d(\C^{n-1})\otimes\mathcal O_{\P^{n-2}}\to\mathcal O_{\P^{n-2}}(d)).$ The isomorphism \eqref{isorest} is $GL_{n-1}$-equivariant. Then 
$$H^0(\bigwedge\nolimits^{p+1}M_{d} \otimes \O_{\P^{n-2}}(b+1+(q-1)d)) =$$ 
$$=H^0\left(\bigoplus_{i = 0}^{p+1} \bigwedge\nolimits^iM'_{d} \otimes (\bigwedge\nolimits^{p+1-i} (\bigoplus_{j \leq d-1}S^j(\C^{n-1})))\otimes \mathcal O_{\P^{n-2}}(b+1+(q-1)d)\right)$$
$$=\bigoplus_{i = 0}^{p+1}\left[H^0\left(\bigwedge\nolimits^iM'_{d} \otimes \O_{\P^{n-2}}(b+1+(q-1)d)\right) \otimes \bigwedge\nolimits^{p+1-i}(\bigoplus_{j \leq d-1}S^j(\C^{n-1}))\right].$$ Keeping only the term corresponding to $i=0$, we have by \eqref{greenk}, \eqref{cohs}, \eqref{greenn}, and \cite[Thm. 3.a.1]{G} an inclusion of $GL_{n-1}$-representations
\begin{equation}\label{godzila}S^{b+1+(q-1)d}(\C^{n-1})\otimes\bigwedge\nolimits^{p+1}(\bigoplus_{j\leq d-1}S^j(\C^{n-1}))\hookrightarrow K_{p,q}(\C^n,b;d).\end{equation} 
In particular, when $q=1$, we obtain an inclusion of $GL_{n-1}$-representations
\begin{equation}\label{lotss}\bigoplus_{0\leq e_p<e_{p-1}<\ldots< e_0\leq d-1}(S^{e_0}(\C^{n-1})\otimes\ldots\otimes S^{e_p}(\C^{n-1})\otimes S^{b+1}(\C^{n-1}))\hookrightarrow K_{p,1}(\C^n,b;d).\end{equation}
\vskip.5cm
The following lemma describes some of the partitions $\lambda$ such that $\mathbb S_{\lambda}$ appears in the decomposition of the LHS of \eqref{lotss}.

\begin{lem}\label{l1} Let $n\geq 2$ and $b\geq0$. Fix $\lambda$ a partition with $|\lambda|= n-1\leq p+2$, and $\lambda_{n-1}>b+1$. We assume that $\lambda$ satisfies $$\lambda\vdash L_0+(b+1),\quad\mbox{with}\quad L_0\geq \frac{p(p+1)}2.$$ For all $0\leq i\leq p$, let \begin{equation}\label{dcond}\left\{\begin{array}{c}e_i:=\lceil\frac{L_i}{p+2-i}+\frac{p+1-i}2\rceil \\ L_{i+1}:=L_i-e_i\end{array}\right.\end{equation}
Then $e_0>e_1>\ldots>e_p\geq 0$ and $L_{p+1}=0$. Assume furthermore that the following conditions hold:
\begin{equation}\label{tcond}\left\{\begin{array}{c}\lambda'_1\geq e_0\\ \lambda'_1+\lambda'_2\geq e_0+e_1\\ \vdots \\ \lambda'_1+\lambda'_2+\ldots+\lambda'_{n-2}\geq e_0+e_1+\ldots+e_{n-3},\end{array}\right.\end{equation}
where $\lambda':=R(b+1,\lambda)$ is the removal of the last $b+1$ visible boxes in $\lambda$. The visible boxes of $\lambda$ are by definition the $\lambda_1$ boxes of the Young diagram of $\lambda$ that have no box directly below them. We order them increasingly from left to right, and then up to down. Then $\mathbb S_{\lambda}$ appears in the decomposition of $$S^{b+1}\otimes\bigotimes\nolimits_{i=0}^pS^{e_i}.$$  
If $e_0\leq d-1$, then $\mathbb S_{\lambda}(\C^{n-1})$ appears in the decomposition of the left term of \eqref{lotss}.
\end{lem}

\noindent The statement of the lemma hints to its algorithmic proof that we leave as an exercise. We are ready to show that the Schur decomposition of $\BK_{p,1}(b;d)$ is asymptotically rich.
\begin{thm}\label{kppositiveb}Fix $p\geq1$ and $b\geq1$. Assume that $p\geq b+1$. Then as $d$ goes to infinity, $$\log c(\BK_{p,1}(b;d)) \in \Theta(\log d), \qquad \log N(\BK_{p,1}(b;d)) \in \Theta(\log d).$$\end{thm}
\begin{proof}
Choose $n$ maximal such that \eqref{greenn} is satisfied. This verifies 
$$n\geq[\sqrt[b+1]{(p+1)\cdot(b+1)!}].$$ Assume $n\geq5$. The plan of the proof of the proposition is as follows: The previous lemma constructs irreducible $GL_{n-1}$-subrepresentations of $K_{p,1}(\C^n,b;d)$. To prove the existence of many distinct types of irreducible $GL_n$-subrepresentations, by the branching rule \cite[Ex. 6.12]{FH}, it is enough to pick a large subset of the $GL_{n-1}$-representations that correspond to Young diagrams that differ pairwise in at least one column by at  least two boxes. We force this by asking that the corresponding partitions satisfy $\lambda_{2i-1}=\lambda_{2i}$ for all $i\leq (n-1)/2$, and $\lambda_{n-1}=\lambda_{n-2}$. Geometrically, we ask that the Young diagram of $\lambda$ sits in a tiling of the plane by $1\times 2$-sized boxes, instead of the standard $1\times 1$ tiling, with the possible exception of using $1\times 3$ boxes for the last three rows when $n-1$ is odd. We say that $\lambda$ has a \textit{twin pattern of length} $n-1$. The Young diagrams of distinct partitions with twin patterns of length $n-1$ (or of lengths of the same parity) automatically differ in at least one column by at least two $1\times 1$ boxes.

\par We now construct an easy-to-count set of partitions $\lambda$ with twin patterns of length $n-1$, also satisfying the conditions of Lemma \ref{l1}. We first ask that 
\begin{equation}\label{res1}\lambda_{n-1}\geq B:=B(p,b,n):=\max\left\{b+2,\frac{\frac{p(p+1)}2+b+1}{n-1}\right\}.\end{equation}
This implies $\lambda_{n-1}>b+1$, and $L_0\geq \frac{p(p+1)}2$. We next impose the restriction 
\begin{equation}\label{res2}\lambda_1\geq\max\left\{\frac{n-3}2,1\right\}\cdot\left(\frac{L_0}{p+2}+\frac{p+3}2\right)\geq \max\left\{\frac{n-3}2,1\right\}\cdot e_0.\end{equation}
This is a manageable condition that together with $\lambda_1=\lambda_2$ implies the relations \eqref{tcond}. The last requirement of Lemma \ref{l1} is $e_0\leq d-1$. For big enough $d$, this is implied by 
\begin{equation}\label{res3}L_0\leq pd.\end{equation}
A bound on $\lambda_3$ so that any further choice for $\lambda_3\geq\lambda_5\geq\ldots\geq\lambda_{2[(n-1)/2]-1}$ satisfies via the twin pattern condition the previous restrictions is:
\begin{equation}\label{res4}B\leq\lambda_3\leq D(\lambda_1,d,p,n):=\min\left\{\lambda_2=\lambda_1,\ \frac{pd-2\lambda_1}{n-3},\ \frac{2(p-n+1)}{(n-2)(n-3)}\lambda_1-\frac{(p+1)(p+2)}{2(n-3)}\right\}.\end{equation}
In the above set, $\lambda_1$ appears because $\lambda$ is a Young diagram. The condition $\lambda_1=\lambda_2$ is part of the tiling restriction. The middle term is explained by \eqref{res3}, and an algebraic manipulation shows that $\lambda_3\leq \frac{2(p-n+1)}{(n-2)(n-3)}\lambda_1-\frac{(p+1)(p+2)}{2(n-3)}$ implies \eqref{res2}.
\par We are reduced to counting partitions $\lambda$ having twin patterns of length $n-1$, satisfying \eqref{res1} and \eqref{res4}. For each fixed $\lambda_1\mbox{ in the interval }\left[B,\frac{pd-(n-3)B}2\right],$ these are counted by the binomial coefficients: $${{[D(\lambda_1,d,p,n)]-\lceil B\rceil+\left[\frac{n-5}2\right]}\choose{\left[\frac{n-5}2\right]}}.$$
When we sum over the range of $\lambda_1$, the result is in $\Theta (d^{\left[\frac{n-3}2\right]})$. This conclusion also holds when $n<5$.  The cases $n\in\{2,3,4\}$ follow immediately from \eqref{lotss}, and the case $n=1$ is impossible, because \eqref{greenn} is satisfied for $n=2$, when $p\geq b+1$. This gives the lower bound needed for the first statement, hence also for the second. The upper bounds follow from Lemma \ref{remtctb} and from Theorem \ref{tensorsymthm}.\end{proof}

\section{Varying $p$}
\addtocontents{toc}{\protect\setcounter{tocdepth}{0}}

A first result in the direction of varying $p$ for syzygy functors is \cite[Cor. 6.2]{EL}. It proves the nonvanishing of $\BK_{p,q}(b;d)$ for fixed $q\geq1$ and sufficiently large $p$. In this section, we study the growth of the complexities of $\K_{p,0}(b;d)$ and of $\K_{p,1}(b;d)$ when we vary $p$.  

\subsection{$\K_{p,0}(b;d)$}

We evaluate the complexity of $\BK_{p,0}(b;d)$ when we fix $b>0$ and $d>2$, and let $p$ grow. It is elementary that $\BK_{p,0}(0;d)=0$, which is why we exclude this trivial case.

\begin{thm}
Fix $b > 0$ and $d>2$. Then as $p$ grows to infinity,
$$\log c(\K_{p,0}(b;d)) \in \Theta (p^{1/2}).$$\end{thm}
\begin{proof}
We argue that, on the one hand, $\log \K_{p,0}(b;d)$ is bounded below by a multiple of $(p^{1/2})$ by using Remark \ref{wexplicit}. On the other, $\log c(\bigwedge^p S^d \otimes S^b)$ is bounded above by a multiple of $(p^{1/2})$, which provides the same upper bound for the cohomology functor.

As observed in the proof of Theorem \ref{kp0bd when d grows}, any $\mathbb S_\lambda$ with $|\lambda|=p+1$ in $\bigwedge^p S^d \otimes S^b$ appears in $\BK_{p,0}(b;d)$. Using Pieri's rule and Remark \ref{wexplicit}, it can be shown that the complexity of $\BK_{p,0}(b;d)$  is at least the number of ways one can partition $p\cdot[(d-1)/2]$ with at most $p$ parts, the logarithm of which is bounded below by a multiple of $(p^{1/2})$ by \cite[p86]{H}. The complexity of $\bigwedge^pS^d \otimes S^b$ is bounded above by the number of partitions of $pd+b$. The logarithm of these is $\Theta(p^{1/2}$) by \cite[p86]{H}.\end{proof}

\subsection{$\K_{p,1}(b;d)$}

In this subsection we look for lower bounds on the complexity of $\BK_{p,1}(b;d)$ when we increase $p$ and $d$. We use the same restriction argument from the previous section, but this time we count partitions $\lambda$ of length $n-1$ that have \textit{almost triplet pattern}, i.e., the second and third parts of $\overline{\lambda}:=\lambda-(1^{|\lambda|})$ are equal, the next three are equal, and so on. We ask that $|\overline{\lambda}|$ is a multiple of 3. The \textit{mold} of a partition $\lambda$ with almost triplet pattern is the partition obtained from $\overline{\lambda}$ by making its first part equal to the second (and third). One sees that irreducible $GL_{n-1}$-representations corresponding to partitions with different molds cannot branch out from the same irreducible $GL_n$-representation after restriction to $GL_{n-1}$.

\begin{lem}
Fix $b\geq0$. Let $n:=n(p)\leq p+1$ be such that $\lim_{p\to\infty}n(p)=\infty$. Then as $p$ grows, for $d\geq3\cdot\left[\frac{p+1}{n-3}\right]+3$, the number of molds of partitions $\lambda$ of length $n-1$ with almost triplet pattern, such that $\mathbb S_{\lambda}\C^{n-1}$ is a $GL_{n-1}$-subrepresentation of 
\begin{equation}\label{big direct sum}
\bigoplus_{ a_0 + ... + a_{d-1} = p+1}  \left(\bigotimes\nolimits^{d-1}_{i = 0} \bigwedge\nolimits^{a_i}S^i\C^{n-1} \right) \otimes S^{b+1}\C^{n-1},
\end{equation}
has its logarithm bounded below by a multiple of $(n^{1/2})$.
\end{lem}

\begin{proof}
We prove a weaker form of the result by asking that $d\geq p+2$. Fix $r\in\{1,2,3\}$ such that $d-r\equiv1\mod 3$. The representation \eqref{big direct sum} contains the nonzero subrepresentation:
\begin{equation}\label{string}\left(\bigwedge\nolimits^{n-1}S^{d-r}\otimes(S^{d-r-1}\otimes S^{d-r-2}\otimes\ldots\otimes S^{d-r-1-(p-n+1)})\otimes S^{b+1}\right)(\C^{n-1}).\end{equation}
By Remark \ref{wexplicit}, for any partition $\mu\vdash\left[\frac{(n-1)(d-r-1)}6\right]$, with $|\mu|\leq l:=\left[\frac{n-1}3\right]$, the partition 
$$\lambda'(\mu):=(1^{n-1})+2(\mu_1,\ \mu_1,\ \mu_1,\ldots,\mu_l,\ \mu_l,\ \mu_l)+\epsilon(d)\cdot(n-1)$$ is such that $\mathbb S_{\lambda'}\C^{n-1}$ is a subrepresentation of $\bigwedge^{n-1}S^{d-r}\C^{n-1}$, where $\epsilon(d)$ is 0 if $d-r-1$ is even, and 1 if $d-r-1$ is odd. By Pieri's rule, $$\lambda(\mu):=\lambda'(\mu)+((\sum_{k=0}^{p-n+1}(d-r-1-k))+b+1)$$ produces a subrepresentation of \eqref{string}. By construction, $\lambda(\mu)$ has almost triplet pattern, and length $n-1$. If $\mu$ and $\mu'$ are different, then $\lambda(\mu)$ and $\lambda(\mu')$ have different molds. At least when $p$ is large enough, the number of $\mu$'s is bigger than the number of partitions of $[(n-1)/3]$, the logarithm of which is in $O (n^{1/2})$ by \cite[p86]{H}.
\par When using the bound $d\geq3\cdot\left[\frac{p+1}{n-4}\right]+3$, the result is proved similarly by considering a subrepresentation of \eqref{big direct sum} of form
$\left(\bigwedge^{n_1}S^{d_1}\otimes\ldots\otimes\bigwedge^{n_m}S^{d_m}\otimes S^{b+1}\right)(\C^{n-1})$, with $m$ as small as possible, such that $\sum_{i=1}^m n_i=p+1$, with $n_i<n$ satisfying some conditions modulo 3, and with $d_1>d_2>\ldots>d_m$ as large as possible, all congruent to 1 modulo 3 and smaller than $d-1$.\end{proof}

As a corollary, we obtain:
\begin{prop}\label{increase d with p}
Fix $b\geq0$. As $p$ goes to infinity, for $n:=\left[\sqrt[b+1]{(p+1)\cdot(b+1)!}\right]$, and $d\geq3\cdot\left[\frac{p+1}{n-3}\right]+3$, there is a positive constant $C$ such that:
$$\log(c(\mathbb{K}_{p,1}(b;d))) \geq C \cdotp (p^{\frac{1}{2(b+1)}}).$$
\end{prop}
\begin{proof}
The choice of $n$ insures that \eqref{greenn} holds, hence by the discussion in the previous subsection, the $GL_{n-1}$-subrepresentations given by the previous lemma appear in the decomposition of $K_{p,1}(\C^n,b;d)$. We conclude by using the restriction argument of the previous section, the discussion in the preamble of this section, and noticing that $n\in \Theta (p^{1/(b+1)})$.\end{proof}

\begin{rmk}\textnormal{
Fixing $d$ and increasing $p$ alone is a more desirable problem to study. In attempting to use the restriction argument, our limitation in showing nontrivial growth for fixed $d$ stems from not knowing if there exist many Schur subfunctors of $\bigwedge^pS^d$ corresponding to partitions of relatively short length and with distinct twin patterns. (In Remark \ref{wexplicit}, all Young diagrams are much too tall.) 
}
\end{rmk}

\noindent\textsc{Department of Mathematics, Princeton University, Princeton, NJ 08544-1000, USA\\
Department of Mathematics, University of Michigan, Ann Arbor, MI 48109, USA
\\ Institute of Mathematics of the Romanian Academy, P. O. Box 1-764, RO-014700,
Bucharest, Romania}
\vskip.2cm
\textsc{E-mail:} afulger@princeton.edu
\vskip.5cm
\noindent\textsc{Department of Mathematics, University of Michigan, Ann Arbor, MI 48109, USA}
\vskip.2cm
\textsc{E-mail:} paulxz@umich.edu

\begin{thebibliography}{99}
\bibitem[BZ]{BZ} A. Berenstein, A. Zelevinski, \textit{Tensor product multiplicities and convex polytopes in
partition space}, J. Geom. Phys. \textbf{5}, no.3 (1989), 453--472
\bibitem[BCR09]{BCR} W. Bruns, A. Conca, T. R\" omer, \textit{Koszul cycles, in combinatorial aspects of Commutative
Algebra and Algebraic Geometry}. Proceedings of the Abel Symposium (2009), 17--32
\bibitem[BCR11]{BCR2} W. Bruns, A. Conca, T. R\" omer, \textit{Koszul homology and syzygies of Veronese subalgebras},
Math. Ann. \textbf{351} (2011), 761--779
\bibitem[BCI]{BCI} P. B\" urgisser, M. Christandl, C. Ikenmeyer, \textit{Even partitions in plethysms}, Journal of Algebra \textbf{328} (2011), 322--329
\bibitem[EEL]{EEL} L. Ein, D. Erman, R. Lazarsfeld, \textit{Asymptotics of random Betti tables}, arXiv:1207.5467v1 [math.AG] (2012)
\bibitem[EL]{EL} L. Ein, R. Lazarsfeld, \textit{Asymptotic syzygies of algebraic varieties}, arXiv:1103.0483v3 [math.AG] (2012), to appear in Invent. Math.
\bibitem[FH]{FH} W. Fulton, J. Harris, \textit{Representation theory, a first course}, Springer--Verlag, New York (1991)
\bibitem[F]{F} W. Fulton, \textit{Young tableaux, With applications to Representation Theory and Geometry}, Cambridge Univ. Press (1997)
\bibitem[G]{G} M. Green, \textit{Koszul cohomology and the geometry of projective varieties, II}, J. Diff. Geom. \textbf{20} (1984),
279--289
\bibitem[Gr]{Gr} F. Grosshans, \textit{Algebraic homogeneous spaces and invariant theory}, Springer Berlin / Heidelberg (1997)
\bibitem[H]{H} G. H. Hardy, \textit{Ramanujan: Twelve lectures on subjects suggested by his life and work}, 3rd ed., Chelsea Publ. Co., New York (1999), Ch.6 83--100, and Ch.8 113--131 
\bibitem[Ho]{Ho} R. Howe, \textit{Asymptotics of dimensions of invariants for finite groups}, J. Algebra \textbf{122} (1989), 374--379.
\bibitem[K]{K} K. Kaveh, \textit{A remark on asymptotic of highest weights in tensor powers of a representation}, arXiv:1002.1984v1 [math.RT] (2010)
\bibitem[KK]{KK} K. Kaveh, A. G. Khovanskii, \textit{Convex bodies associated to actions of reductive groups}, arXiv:1001.4830v2 [math.AG] (2012)
\bibitem[Kh]{Kh} A. G. Khovanskii, \textit{Newton polyhedron, Hilbert polynomial and sums of finite sets}, Funct. Anal. Appl.
\textbf{26} (1993), 276--281
\bibitem[LM]{LM} R. Lazarsfeld, M. Musta\c t\u a, \textit{Convex bodies associated to linear series}, Ann. Sci. \'Ec. Norm. Sup\'er. (4) \textbf{42}, no.5 (2009), 783--835
\bibitem[M]{Man}L. Manivel, \textit{Applications de Gauss et pl\' ethysme}, Ann. de l'instituit Fourier, \textbf{47}, no.3 (1997), 715--773
\bibitem[MM]{MM}L. Manivel, M. Michalek, \textit{Effective constructions in plethysms and Weintraub's conjecture}, 	arXiv:1207.5748v1 [math.RT] (2012)
\bibitem[N]{New} M. J. Newell, \textit{A theorem on the plethysm of S-functions}, Quart. J. Math. \textbf{2} (1951), 161--166
\bibitem[O]{O} A. Okounkov, \textit{Brunn--Minkowski inequality for multiplicities}, Invent. Math. \textbf{125}, no.3 (1996), 405--411
\bibitem[OP]{OP} G. Ottaviani, R. Paoletti, \textit{Syzygies of Veronese embeddings}, Compos. Math. \textbf{125}
(2001), 31--37
\bibitem[OR]{OR} G. Ottaviani, E. Rubei, \textit{Resolutions of homogeneous bundles on $\P^2$}, Ann. Inst. Fourier (Grenoble) \textbf{55}, no.3 (2005), 973--1015
\bibitem[P]{P} R. Paoletti, \textit{The asymptotic growth of equivariant sections of positive and big line bundles}, Rocky Mountain J. Math. 35, \textbf{6} (2005), 2089--2105
\bibitem[Ra]{Ra} C. Raicu, \textit{Representation stability for syzygies of line bundles on Segre--Veronese varieites}, to appear
\bibitem[R]{R} E. Rubei, \textit{A result on resolutions of Veronese embeddings}, Ann. Univ. Ferrar Sez. VII \textbf{50} (2004),
151--165
\bibitem[Sch]{Sch} F. Schreyer, \textit{Syzygies of canonical curves and special linear series}, Math. Ann. \textbf{275} (1986), 105--
137
\bibitem[Sn]{Snow} A. Snowden, \textit{Syzygies of Segre embeddings and $\Delta$-modules}, arXiv:1006.5248v4 [math.AG] (2011)
\bibitem[S]{S} R. P. Stanley, \textit{Enumerative combinatorics, Vol.2.}, Cambridge Univ. Press, Cambridge (1999)
\bibitem[TZ]{TZ} T. Tate, S. Zelditch, \textit{Lattice path combinatorics and asymptotics of multiplicities of weights in tensor powers}, J. Funct. Anal. \textbf{217}, no.2 (2004), 402--447
\bibitem[W]{Wei} S. H. Weintraub, \textit{Some observations on plethysms}, J. Algebra \textbf{129} (1) (1990), 103--114
\bibitem[Z]{Z} X. Zhou, \textit{Effective non-vanishing of asymptotic adjoint syzygies}, arXiv:1204.0123v1 [math.AG] (2012), to appear in Proc. of AMS. 

\end{thebibliography}
\end{document}